\documentclass{amsart}
\usepackage{amsmath,amssymb,amsthm,pinlabel,tikz,hyperref,mathrsfs,color}
\usepackage{verbatim}
\usepackage{graphicx,overpic}

\renewcommand{\emph}{\textbf}

\newcommand{\nc}{\newcommand}
\nc{\dmo}{\DeclareMathOperator}
\dmo{\ra}{\rightarrow}
\dmo{\N}{\mathbb{N}}
\dmo{\Z}{\mathbb{Z}}
\dmo{\Q}{\mathbb{Q}}
\dmo{\R}{\mathbb{R}}
\dmo{\C}{\mathcal{C}}
\dmo{\AC}{\mathcal{AC}}
\dmo{\Mod}{Mod}
\dmo{\PMod}{PMod}
\dmo{\B}{B}
\dmo{\PB}{PB}
\dmo{\I}{\mathcal{I}}
\dmo{\el}{\ell_{\C}}
\dmo{\NN}{\mathcal{N}}
\dmo{\rk}{rk}
\nc{\s}{\mspace{1mu}}
  
\nc{\nt}{\newtheorem}

\nt{theorem}{Theorem}

\newtheorem{thm}{{\bf Theorem}}[section]
\newtheorem{lem}[thm]{{\bf Lemma}}
\newtheorem{cor}[thm]{{\bf Corollary}}
\newtheorem{prop}[thm]{{\bf Proposition}}

\newtheorem{claim}[thm]{Claim} 
\newtheorem{remark}[thm]{Remark}

\newtheorem{definition}[thm]{Definition}
\nt{example}[thm]{Example}
\numberwithin{equation}{section}

\title{The Farey tree and fibrations of the Whitehead link complement}

\author{Tali Pinsky}
\address{The Technion, Haifa and Monash University, Melbourne}
\email{talipi@technion.ac.il}
\date{\today}

\begin{document}
\begin{abstract}
We show a curious connection between
an example of dynamical forcing, where an existence of one periodic orbit for a dynamical system forces the existence of other periodic orbits, and the existence of different fibrations of a 3 dimensional manifold.

Specifically, we establish a one to one correspondence between fibrations of the Whitehead link complement and the set of ``simple orbit pairs" on the torus, defined in \cite{pinskyWajnryb}. The set of orbits is equipped with a complete order coming from a dynamical forcing relation, that corresponds to the Farey tree order of rational numbers. This corresponds to the order of the rational point on a fibered cone of the Whitehead link.
For the proof we obtain an explicit description of all monodromies of the different fibrations of the Whitehead link complement.
\end{abstract}

\maketitle

\section{Introduction}
The first example of a \emph{forcing relation} of periodic orbits of a dynamical system is Sharkovskii's proof, that there's a complete order relation $\succeq$ on the natural numbers so that the existence of an orbit of order $k$ for a continuous function $f:\R\to\R$ implies $f$ has an orbit $n$ for any $n$ smallwe than $k$ ($k\succeq n$) \cite{sharkovskii}. Strikingly, $3$ is the largest number and the existence of an orbit of order 3 implies existence of an orbit of period $n$ for any natural $n$. This case is often referred to as `period 3 implies chaos' \cite{LiYork1975period3}. 

It turns out that in dimension two, similar to the one dimensional case, one can tighten a given diffeomorphism $h$ with respect to a fixed periodic orbit $x=\{x_1,\dots,x_n\}$, to obtain a \emph{minimal representative} with the least possible number of periodic orbits. The minimal representative is the Thurston Nielsen canonical form, and all its periodic orbits are thus forced by the orbit $x$ (see \cite{boyland1994methods} for a review). 
 Methods for the 2-d case and specific examples have been obtained by Matsuoka \cite{Matsuoka1986analogue}, Boyland \cite{Boyland1992Rotation}, Jiang \cite{Jiang1993FixedPoints}, Handel \cite{Handel1997forcing3braids}, Los\cite{Los1997ForcingRelation}, Carvalho and Hall \cite{CarvalhoHall2002HorseshoeBraidType}, Kin \cite{Kin2008ForcingPartialOrder}, and many others.
This 2 dimensional procedure is explained in Section \ref{sec:Pairs}.

A homeomorphism of the 2-torus is of shear type if it is homotopic to a Dehn twist. 
Pinsky-Wajnryb \cite{pinskyWajnryb}, introduce the notion of a ``simple pair" for  two coexisting periodic orbits of a specific type, and proved a forcing result: Each simple pair forces infinitely many other periodic orbits that are also simple pairs (See Section \ref{sec:Pairs}).

Fixing a homeomorphism of the torus that has such a pair of orbits, one can puncture each periodic point in the simple pair and obtain a homeomorphism of the punctured torus. The resulting homeomorphism is of pseudo-Anosov type. 
In particular, its mapping torus is a hyperbolic  3-manifoldand a direct computation shows it is the Whitehead link exterior. 

Our main theorem is the following.

\begin{thm}\label{thm: 1-1 monodromies and simple pairs}
    There is a one-to-one correspondence between the set of fibrations of the Whitehead link exterior and 
the set of simple pairs corresponding to Farey neighbors $p,q\in\mathbb{Q}\cap[0,1]$. 
\end{thm}

We explicitly describe all monodromies of the different fibrations of the Whitehead link complement and give a train track for their monodromies.

\subsection*{Acknowledgments} 
I am grateful to Eiko Kin for suggesting this project and for many helpful discussions.

\section{The Farey diagram}

The Farey diagram is a way to give structure for all rational numbers, which we will use in the proof of Theorem~\ref{thm: 1-1 monodromies and simple pairs}. We consider only rational numbers between 0 and 1.

\begin{definition}
 Two rational numbers $\frac{p}{q}$ and $\frac{r}{s}$ are called \emph{Farey neighbors} if $|ps-rq|=1$. \\
 For Farey neighbors  $\frac{p}{q}$ and $\frac{r}{s}$, define their \emph{Farey child} to be the rational number $\frac{p}{q}\oplus\frac{r}{s}=\frac{p+r}{q+s}$.
\end{definition}

One constructs the Farey diagram for rationals in $[0,1]$ starting with $0=\frac{0}{1}$ and $1=\frac{1}{1}$. As these are Farey neighbors, draw an edge connecting them. Then consider their Farey child $\frac{1}{2}$.
A child $\frac{p+r}{q+s}$ is always a Farey neighbor of each of its parents (taking into account the parents are neighbors), thus connect $\frac{1}{2}$ with edges to both $\frac{0}{1}$ and $\frac{1}{1}$.

Next, consider the Farey child $\frac{0}{1}\oplus\frac{1}{2}=\frac{1}{3}$ and the Farey child $\frac{1}{1}\oplus\frac{1}{2}=\frac{2}{3}$, and connect each with edges to each of its parents.

Continue to construct other rational numbers by adding the child for any edge added in the previous step, and adding the edges connecting them to both their parents. 
It is not hard to prove that the diagram eventually reaches all rational numbers. Thus, given a pair of Farey neighbors, the set of all their descendants is exactly the set of all rational numbers between them along the real line.

\begin{figure}[!ht]
\centering
\begin{overpic}[height=5.5cm]{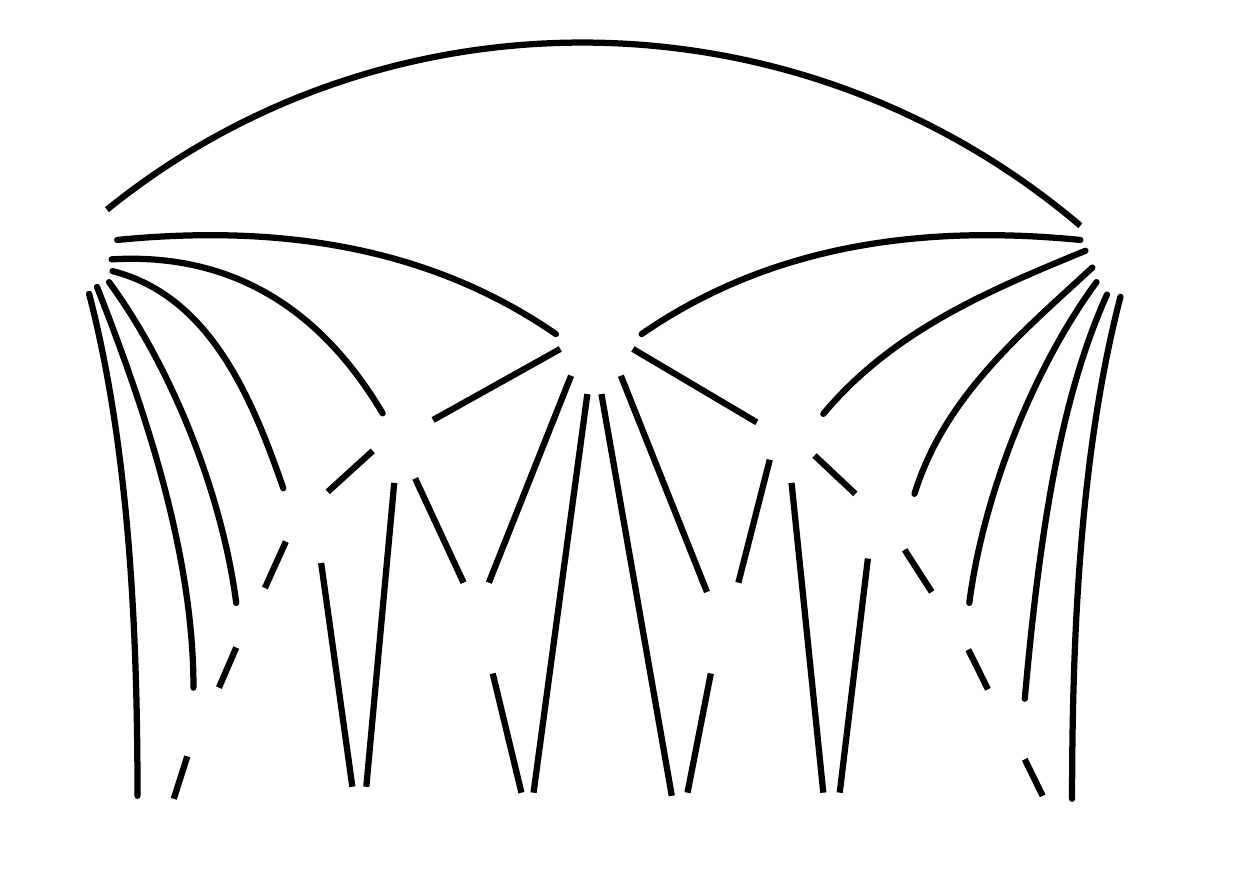}
    \put(5,50){$\frac{0}{1}$}
    \put(88,50){$\frac{1}{1}$}
    \put(47,42){$\frac{1}{2}$}
    \put(31,34){$\frac{1}{3}$}
    \put(62.5,34){$\frac{2}{3}$}
    \put(23.3,28){$\frac{1}{4}$}
    \put(70,28){$\frac{3}{4}$}
    \put(19,20){$\frac{1}{5}$}
    \put(37,19.5){$\frac{2}{5}$}
    \put(57,19.5){$\frac{3}{5}$}
    \put(75.3,20){$\frac{4}{5}$}
    \put(15,11){$\frac{1}{6}$}
    \put(80,11){$\frac{5}{6}$}
    \put(11,2){$\frac{1}{7}$}
     \put(28,2){$\frac{2}{7}$}
      \put(41.5,2){$\frac{3}{7}$}
       \put(53,2){$\frac{4}{7}$}
        \put(66,2){$\frac{5}{7}$}
         \put(84,2){$\frac{6}{7}$}
\end{overpic}
\caption{A part of the Farey diagram showing a few  Farey neighbors with small denominators between  0 and 1.}
\label{fig_tree}
\end{figure}
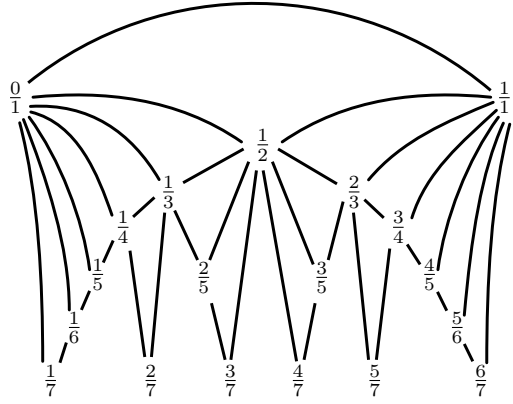

\section{Simple Pairs}\label{sec:Pairs}
To specify an orbit means for us to specify the manner in which the points of the orbits move around the manifold up to isotopy. Thus in dimension two it is not enough to mention the period of the orbit (as it is for Sharkovskii), and a periodic orbit $x$ for a diffeomorphism $f:S\to S$ is specified either as a braid, or equivalently as an element of the mapping class of ${S\setminus x}$. This is also called the Strong Nielsen type of the orbit, and this is what we mean by the word orbit in this paper. 

Once an orbit is specified, the general method to obtain orbits that are forced by it is to find its Thurston-Nielsen canonical form that is either periodic, pseudo Anosov, or reducible using the Bestvina-Handel algorithm. 
By works of Boyland \cite{boyland1994methods}, Jiang \cite{Jiang1993FixedPoints} and others, this representative, after corrections on the boundary, is \emph{minimal}.  i.e., any periodic orbit that exists for the canonical form exists (up to isotopy) for any diffeomorphism in the same isotopy class (or mapping class). Thus all orbits appearing in the canonical form are forced by the orbit we started with. These orbits typically can be found by a symbolic dynamics that is easy to establish for the canonical form. For example if one uses the Bestvina-Handel Algorithm \cite{BestvinaHandel1995TrainTracks} Or de Carvalho-Hall Pruning theory \cite{deCarvalhoHall2001Pruning}, one obtains the symbolic dynamics at the same time as the canonical form.

\begin{definition}\label{def:simple orbit}
    Let $h$ be a homeomorphism of $T^2$. An orbit $x=\{x_1,\dots x_n\}$ is called \emph{simple} is $T^2\setminus x$ has a deformation retract onto a graph $G\subset T^2$ So that
    $G$ is composed of $n$ closed loops $\beta_1,\dots,\beta_n$ that are mutually disjoint, and a single loop $\alpha$ intersecting each of $\beta_i$ at a single point, and so that the image under $h$ of a loop $b_i$ is a loop $b_j$, and the image of the loop $\alpha$ is $T_{b_i}(a)$, i.e. it is mapped to  itself with a Dehn twist along a loop $b_i$ for some $1\leq i\leq n$.

    The loops $b_i$ are called \emph{vertical}, and the loop $\alpha$ is called \emph{horizontal}. Since for any $i$ the intersection $\alpha\cap b_i$ is a single point, this gives a framing for $T^2$ (i.e. a basis for the homology).
\end{definition}

\begin{definition}\label{def:simple_pair}

Let $h$ be a homeomorphism of $T^2$. Consider two orbits $x=\{x_1,\dots x_n\}$ and $y=\{y_1,\dots y_m\}$ so that $n\leq m$. The pair $\{x, y\}$ is called a \emph{simple pair} 
if:

\begin{enumerate}
\item There exists a graph that is a deformation retract of $T^2\setminus(x\mspace{1mu} \cup\mspace{1mu} y)$ as in Figure~\ref{fig:graph and image} as follows: One loop $\alpha$ that we again call `horizontal'; A family of $m$ `vertical' loops $\beta_1,\dots,\beta_m$, not intersecting each other and each intersecting the horizontal loop $\alpha$ at a single point, so that each region in the complement of $\alpha\cup\bigcup_{i=1}^m\beta_i$ contains a single point of the $y$ orbit and at most a single point of the $x$ orbit.
\item The image of the graph under $h$, up to isotopy relative to the two orbits is as in Figure~\ref{fig:graph and image}, namely: Each vertical loop is sent to another vertical loop except one vertical loop that we denote by $b$, which is sent to another vertical loop plus a boundary of a disk punctured once by the $x$ orbit; The vertical loop $a$ is mapped to $T_{h(b)}(a)$, i.e. to itself with a Dehn twist along the image of $b$.
\end{enumerate}

Denote a simple pair of orbits by $x\vee y$. Note that the simple pair is not simply a union of two simple orbits, as an orbit isotopic to $x$ might not form a simple pair with $y$.

\begin{figure}[ht!]
\centering
\includegraphics[height=4cm]{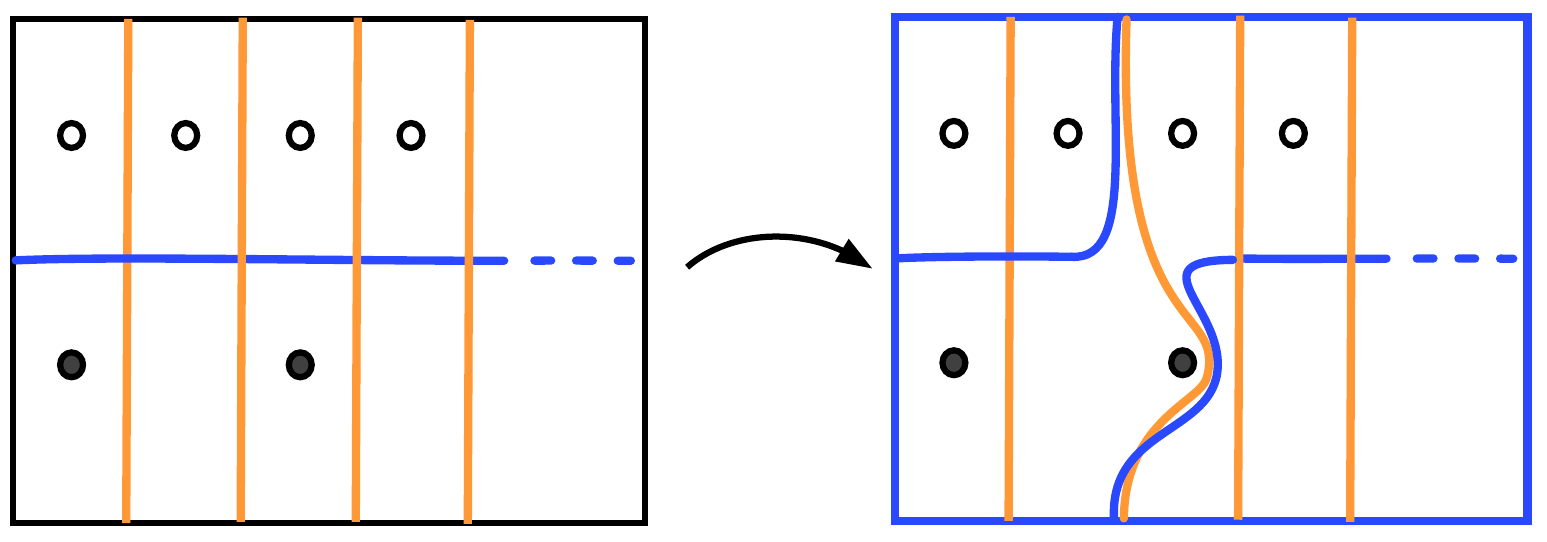}
\caption{A graph that is a spine for the complement of a simple pair of orbits for a homeomorphism $h$ on $T^2$ and their image under the homeomorphism $h$. The image of one of the vertical loops $b_i$ is another vertical loop $b_j$ up to a once punctured disk, that is depicted on the right of $b_j$ but can also be on its left. The horizontal loop $\alpha$ is mapped to $T_{h(b_i}(\alpha)$.}
\label{fig:graph and image}
\end{figure}
\end{definition}

\begin{remark}
    Consider the action of $h$ on the graph in Figure~\ref{fig:graph and image}, given on the right side of the figure. If one ignores the orbits $x$ and $y$ then all vertical circles of the graph are homotopic to a single vertical circle $\beta$ and the action of $h$ on $\beta$ is trivial, while the action of $h$ takes the horizontal circle $a$  to $T_\beta(a)$. Thus, if there exists a simple pair for a homeomorphism $h$, $h$ has to be of shear type, i.e. it is homotopic to a Dehn twist along $\beta$.
 \end{remark}

\begin{definition}(see \cite{DOEFF_1997})
    Let $T^2$ be a two dimensional torus with a basis for homology $\alpha,\beta$, and let $h$ be a shear homeomorphism of $T^2$ in that basis. Let $\tilde h$ be a lift of h to the universal cover $\R^2$, where the unit square is a fundamental domain. Then for every $p\in\mathbb{R}^2$ and any $v\in\mathbb{Z}^2$,
    \[
\tilde h (p+v)=\tilde h(p)+\begin{pmatrix}
1 & 0\\
1 & 1 
\end{pmatrix}v.
    \]
    
    For any periodic point $x$ of $h$ of period $p$, $\tilde h^p$ maps any lift $\tilde x$ of $x$ the same integer number $q$ along the horizontal axis away from $\tilde x$, in a standard choice of axis (and $\tilde x$ is possibly mapped some integer number along the vertical axis as well). We can then define the rotation number of $x$ to be  
    $r(x)= \frac{q}{p} \mod 1$.
    For a simple orbit or pair, we define the rotation number to be corresponding to the basis $\{\alpha,\beta\}$ induced by the oriented corresponding graph.
     \end{definition}

    \begin{lem}
    The rotation number $r(x)$ is well defined up to sign, and does not depend on the chosen graph and lift $\tilde h$ of $h$.
 \end{lem}

\begin{proof}
Assume first that we have a simple orbit $x$ for $h$, with a given graph of vertical and horizontal loops, and we have two different lifts, $\tilde h$ and $\tilde g$ of the diffeomorphism $h$.
Then $\tilde g(\tilde y)=\tilde h(\tilde y)+ \kappa$ where $\kappa\in\mathbb{Z}^2$, and this is true for any point $y\in\mathbb{R}^2$. 
Choose any preimage $\tilde x_1$ of a point $x_1\in x$.
$\tilde g^n(\tilde y)=\tilde h^n(\tilde y)+ n\kappa$, and thus
\[
r_{\tilde g}(x) =r_{\tilde h}(x)+nk_1\ \mod 1=r_{\tilde h}(x)\ \mod 1
\]
    
    Next, suppose an orbit of $h$ is simple with respect to two different graphs, so that both have the same structure as in 
    Definition \ref{def:simple_pair}. The definition only depends on the homology classes of the loops in the graph. If the two graphs are non homologous, this is equivalent to choosing a different horizontal/vertical direction, or equivalently a different basis of $\mathbb{R}^2$. Consider the two bases $\{a_1,b_1\}$ and $\{a_2,b_2\}$, and a lift $\tilde h$ of $h$ fixing the origin. As $b_1$ and $b_2$ are both integral and are invariant under $\tilde h$, we have $b_2=\pm b_1$. As we have $a_i\to a_i+b_i$,  we have $a_2=\pm(a_1+mb_1)$. Thus, the projections onto $a_1$ and $a_2$ agree up to sign, and the statement follows.
\end{proof}

On the other hand, a rotation number uniquely defines a simple orbit, as fixing a torus with a horizontal and vertical directions, it determines the order of the orbit and the number of vertical loops to the right (along the horizontal direction) each vertical loop is mapped to. Thus, we may specify a simple orbit using a rational number in $[0,1]$.

\begin{claim}
    the orbit $x$ corresponding to  $\frac{p}{q}\in[0,1]$ is equivalent, in the sense that the mapping class element of the punctured torus is conjugate, to the orbit $x'$ corresponding to $1-\frac{p}{q}$.
    \end{claim} 

\begin{proof}
    A rotation by $\pi$ around an intersection point of $\alpha$ with one of the vertical loops takes $x$ to $x'$ showing they are equivalent. It takes vertical loops to vertical loops and the horizontal loop to itself, proving that the mapping class elements on the complements are conjugate.
\end{proof}

The author is grateful to Bronek Wajnryb for suggesting the following example.

\begin{example}\label{ex:non simple orbit}
There exists a homeomorphism $h$ of $T^2$ that has a single periodic orbit that is a non-simple orbit with rotation number $\frac{1}{2}$. This shows there exist non-simple orbits for shear homeomorphisms, and that the rotation number does not characterize the orbit.

\end{example}

\begin{proof}
 Consider a two dimensional torus $T^2$, and fix a horizontal-vertical framing, with coordinates by $(s,t)\in [0,1]\times[0,1]$. Define a homeomorphism $g$ as follows: $g$ maps the left half  $\{0\leq s\leq \frac{1}{2}$ of the torus to the right half.
 Points with $(s,t)$ with $s$ coordinate between $0$ and $\frac{1}{2}-\frac{1}{10}$ are mapped to $(s',t)$  where $s'= s+1/2 \mod 1$. Points with $s$ coordinate between $\frac{1}{2}-\frac{1}{10}$ and $\frac{1}{2}$ are mapped to $(s',t+h(s))$ where $h(s)$ slowly increases over this segment from $0$ to an  irrational angle $\alpha$, so that $0<\alpha<\frac{1}{10}$.
 For the right half of the torus, a point $(s,t)$ is mapped to a point $(s'+\delta(s), t+h(s)$) where $h(s)$ is a continuous map monotonically increasing from $\alpha$ for $s=\frac{1}{2}$ to $1$ for $s=1$, so that $h(3/4)=1/2$. The function $\delta(s)$ is a scaling of the sine function, so that vertical loops with $\frac{1}{2}\leq s \leq \frac{3}{4}$ move to the left from the corresponding loops  $l_{s'}$ and vertical loops with $\frac{3}{4}\leq s \leq 1$ move to the right from the corresponding loops $l_{s'}$, but is small enough so that the function stays one to one. All together $g$ maps vertical loops to vertical loops, and the loops  $L=\{l_0,l_\frac{1}{4}, l_\frac{1}{2}, l_\frac{3}{4}\}$ are invariant as a set. 

\begin{figure}[ht]
    \centering
    \begin{overpic}[width=5cm]{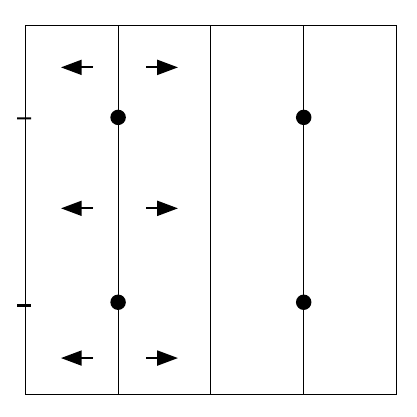}
    \put(-1,26){$\frac{1}{4}$}
    \put(-1,70){$\frac{3}{4}$}
    \put(27,-2){$\frac{1}{4}$}
    \put(49,-2){$\frac{1}{2}$}
    \put(72,-2){$\frac{3}{4}$}
     \end{overpic}
    \caption{The torus $T^2$, some vertical loops used for the definition of $g$ and $h$, and the unique periodic orbit with four points and rotation number $\frac{1}{2}$.}
    \label{fig:non simple}
\end{figure}
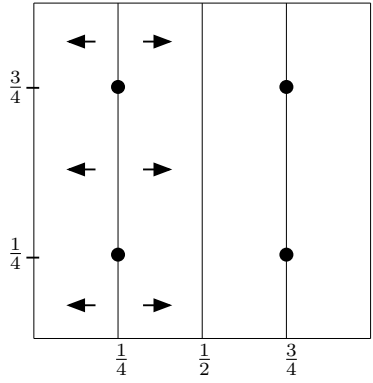

It is easy to see the only periodic orbits for $g$ are on the loops $l_\frac{1}{4}$  and $l_\frac{3}{4}$. This is because odd powers of $g$ switch the two halves of the torus and thus there is no periodic orbit with odd order. For nonzero even powers of $g$, every point on a vertical loop not in $L$ is mapped to vertical loops closer and closer to $l_0$ or $l_\frac{1}{2}$, and thus cannot be periodic. Points on $l_0$ and $l_\frac{1}{2}$ are mapped back to the same vertical loop rotated by some power of the irrational angle $\alpha$, and points on $l_\frac{1}{4}$  and $l_\frac{3}{4}$ are mapped under $g^2$ back to the same vertical loop rotated by $1/2$, and by $g^4$ back to themselves.

Finally, one can add a term that is a scaling of $\sin^2(t)$  for points on $l_\frac{3}{4}$ (and slowly decreasing on vertical loops in a small collar neighborhood of this loop) to leave a unique periodic orbit there, and have all other points move upwards and limit onto this orbit from below. This results in a homomorphism $h$ on $T^2$, of shear type, that has a unique periodic orbit with rotation number $\frac{1}{2}$ which is clearly not the simple orbit of rotation number $\frac{1}{2}$, as it has period 4 instead of 2.
\end{proof}

\begin{lem}
For a simple pair $x\vee y$, the rotation numbers are Farey neighbors.
\end{lem}

\begin{proof}
 Let $x\vee y$ be a simple pair, where $x$ has rotation number $p_1/q_1$, and $y$ has rotation number $p_2/q_2$. Without loss of generality, assume $x$ has larger period i.e. $q_1>q_2$. 

 Then there are $q_1$ rectangles in the complement of a graph as in Figure~\ref{fig:graph and image}, each containing a point of $x$, and $q_2$ out of those contain in addition a point of $y$. 

 Choose the point $y_0$ of the orbit $y$ that is within the "bump" in the right side in Figure~\ref{fig:graph and image}, and choose $x_0$ to be the point above $y_0$ in the same rectangle $R_0$. 

 Applying the homeomorphism $h$ moves $x_0$ and $y_0$ together to some other rectangle containing two points. After applying $h$ $q_2$ times, $y_0$ comes back to itself after circling the torus horizontally  $p_2$ times. The point $x_0$ shifts by $\frac{p_1}{q_1}$ rectangles every application of $h$, and after $q_2$ applications lands one point of the $x$ orbit to the right of $x_0$, in the case of Figure~\ref{fig:graph and image}, or to the left in the other case where the disk is on the left, while applying the homeomorphism $q_1$ times $x_0$ comes back to itself exactly.
 
 Therefore, \[
q_2\frac{p_1}{q_1}=p_2\pm\frac{1}{q_1},
 \]
 and $\frac{p_1}{q_1}$ and $\frac{p_2}{q_2}$ are Farey neighbors as required.
\end{proof}

\begin{definition}
    Let $x$ and $y$ be two classes of periodic orbits on a surface $S$. We say $x$ is forced by $y$ if any homeomorphism $h$ of $S$ containing $y$ as a periodic orbit (recall that this implies $S$ has a specific behavior  up to isotopy on $S\setminus y$) also contains $x$ as a periodic orbit. See \cite{boyland1994methods}) We denote this by $y\succeq x$.
\end{definition}

\begin{thm}[Pinsky-Wajnryb \cite{pinskyWajnryb}] 
Let $\frac{p}{q} \vee \frac{r}{s}$ be a simple pair for a homeomorphism $h$ of the two dimensional torus.  Then for any Farey neighbors $\frac{\alpha}{\beta}$ and $\frac{\gamma}{\delta}$ that are between $\frac{p}{q}$ and $\frac{r}{s}$, 
\[
\frac{p}{q} \vee \frac{r}{s} \succeq \frac{\alpha}{\beta}\vee\frac{\gamma}{\delta},\ \frac{p}{q} \vee \frac{r}{s} \succeq \frac{\alpha}{\beta}.
\]
Thus an edge in the Farey tree forces all its offspring and offspring neighbors.
\end{thm}

\begin{remark}
It follows in greater generality by Doeff \cite{DOEFF_1997} that if a shear homeomorphism contains orbit with two distinct rotation numbers, it contains a periodic orbit with any rational rotation number between them (but by Example \ref{ex:non simple orbit} the rotation number does not uniquely determine an orbit).
\end{remark}

\section{A fibered cone for the Whitehead link exterior}

In \cite{Thurston1986Norm}, Thurston proves that a rational interior point of any face on the Thurston norm ball of a three dimensional manifold corresponds to a fibration of the manifold. He computes the norm ball for the Whitehead link complement as an example, and this results in the ball $B$ shown in Figure~\ref{fig:Thurston_ball}.

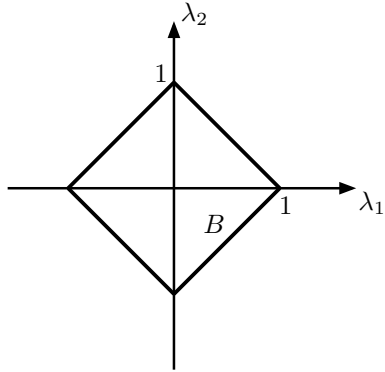
\begin{figure}[!ht]
    \centering
    \begin{overpic}[width=5cm]{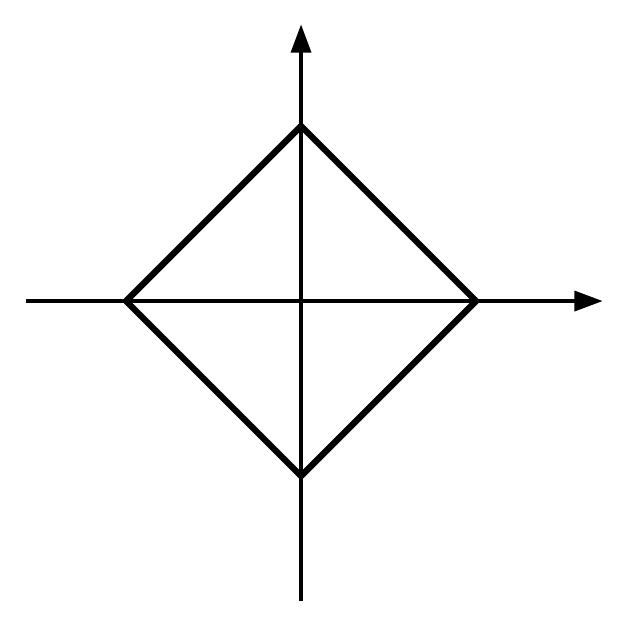}
    \put(97,46){$\lambda_1$}
    \put(50,96){$\lambda_2$}
    \put(56,40){$B$}
    \put(76,45){1}
    \put(43,80){1}
     \end{overpic}
    \caption{The Thurston norm ball for the Whitehead link complement.}
    \label{fig:Thurston_ball}
\end{figure}

The axes $\lambda_1$ and $\lambda_2$ in the figure represent the longitudes on the two components of the link. Thurston notes that it is easy to find a twice punctured disk with boundary say $\lambda_1$ and two meridional components on the other link component, by taking a disk bounded by the first component that the second component punctured twice as in Figure \ref{fig_whitehead}. Then, one may attach an annulus connecting the two meridional components and contained in the boundary of a small neighborhood of the second component, to obtain a once punctured torus with a single boundary $\lambda_1$, and similarly for $\lambda_2$.

\begin{figure}[ht!]
\centering
\includegraphics[height=1.5in]{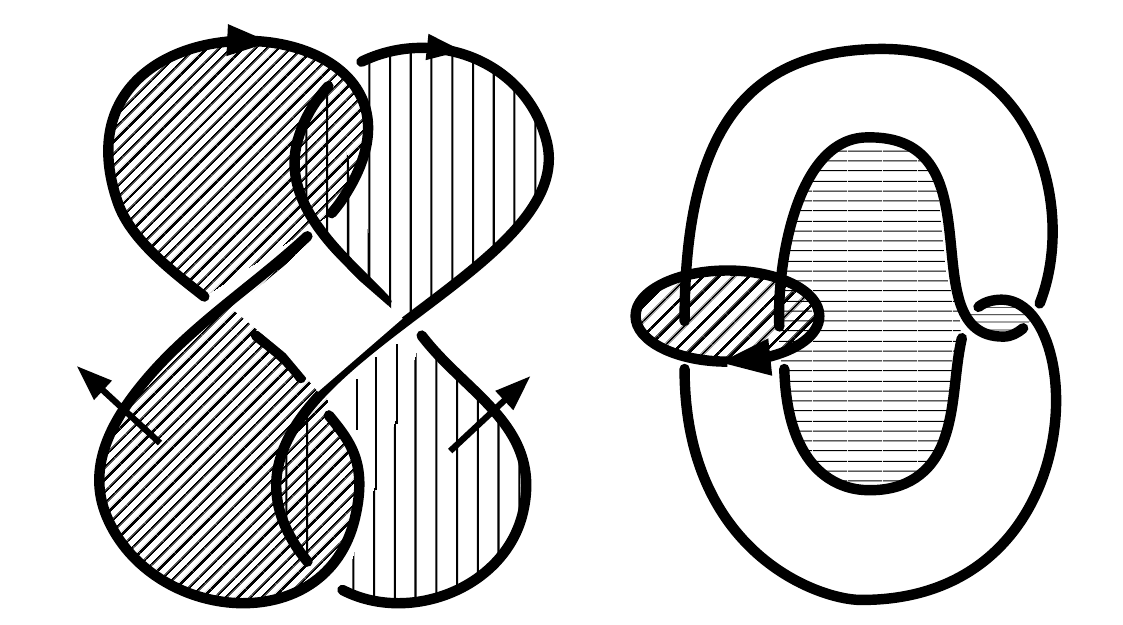}
\caption{Two diagrams of the Whitehead link. 
The norm minimizing surfaces $P_1$ and $P_2$ corresponding to $\mathfrak{a}$ and $\mathfrak{b}$, each a thrice punctured sphere, in two representations of the Whitehead link.}
\label{fig_whitehead}
\end{figure}

The link has symmetries fixing one component while changing the orientation of the other, thus to obtain all monodromies it is sufficient to consider positive integers $n$ and $m$.
Let $\mathscr{C}$ be the fibered cone corresponding to the positive quadrant in the homology space of the Whitehead link exterior. 
Each element $\alpha \in \mathscr{C}$ is represented by $\alpha= x \lambda_1 + y \lambda_2$ 
for some $x, y \in {\R}_{\ge 0}$. 
We write $\alpha= (x,y)$ in this case. 
The Thurston norm $\|\alpha\|$ is equal to $ x+y$ as $\chi( x \lambda_1 + y \lambda_2)=x\cdot\chi(\lambda_1)+y\cdot\chi(\lambda_2)=-(x+y)$.

Thus the face is a straight diagonal segment, and a rational point in an interior of the face is a rescaling of $n\lambda_1+m\lambda_2$, with $n,m$ positive, nonzero, and can be assumed to be relatively prime. A fiber corresponding to this fibration is obtained by taking the Haken sum of $n$ copies of the once punctured torus with boundary $\lambda_1$ and $m$ copies of the torus with boundary $\lambda_2$. 
Thus, if $\alpha= (k, \ell) \in \mathscr{C}$ is a primitive integral class, then 
the fiber $S_{\alpha}$ is homeomorphic to $\Sigma_{1, k+ \ell}$, a torus with $k+\ell$ punctures.

By the symmetry of the Whitehead link interchanging the two components, 
the inverse $\Phi_{(k,\ell)}^{-1}$ of the monodromy $\Phi_{(k,\ell)}$ of the fibration on $W$ associated to $(k,\ell) \in \mathscr{C}$ 
is conjugate to the monodromy $\Phi_{(\ell, k)}$ of the fibration on $W$ associated to $(\ell,k) \in \mathscr{C}$.

\section{The Proof}
In this section we prove our main theorem, namely that there is a one to one correspondence between the set of all simple pairs and the set of all monodromies of fibrations of the Whitehead link complements. We start by showing an inclusion in one direction.

\begin{prop}[Rolfsen \cite{Rolfsen}]\label{prop:The 1-1 case}
    The monodromy corresponding to the fibration $\lambda_1+\lambda_2$ is the homeomorphism given by the simple pair $\frac{0}{1}\vee\frac{1}{1}$.
\end{prop}

\begin{remark}
The above monodromy can also be computed in a number of different methods. In Appendix B of \cite{KinRolfsen2018}, the monodromy is computed using Murasugi sums, and it can also be computed from the monodromy  of an appropriate fibration of the magic manifold given in \cite{Kin2015Magic} by Dehn filling.
\end{remark}

\begin{proof}
We take an approach due to Goldsmith \cite{Goldsmith1975Symmetric}, used in page 336 in Rolfsen's book \cite{Rolfsen} and also in Birman and Williams \cite{BirmanWilliamsII}. First, notice that the Whitehead link is a double cover of one component of the link showed on the left of Figure~\ref{fig:disks}, branched along the other component (which is trivial, and depicted as a straight line. This proves it is fibered as the link is symmetric. We expand on the details for completeness.

Depicting the link when the $K$ component is arranged along the $z$ axis as in Figure~\ref{fig:disks} (viewing the projection onto the $xy$ plane from above), shows one component intersects a disk bounded by the other component 4 times. Cutting along this disk to produce the branched cover results in the Whitehead link shown on the right side of the figure.

\begin{figure}[ht]
    \centering
    \includegraphics[width=4cm]{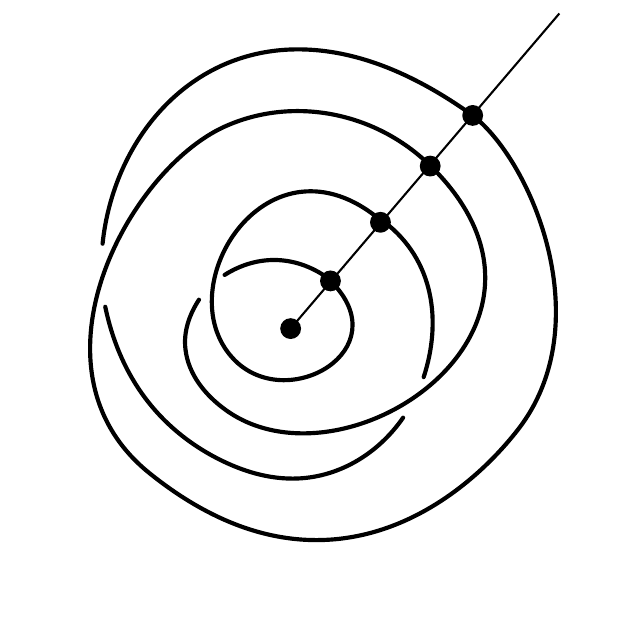}
    \includegraphics[width=4cm]{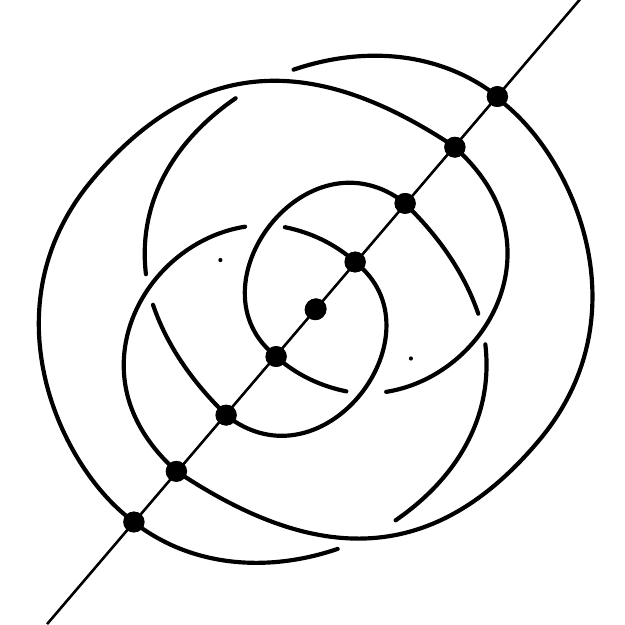}
    \caption{Taking the branched cover along one component takes the second component to the Whitehead link as shown on the right.}
    \label{fig:disks}
\end{figure}

The trivial link on the left of Figure~\ref{fig:disks} is symmetric, thus the monodromy computed for one component (with the other as the braid axis) is identical to the one computed for the other component.

The monodromy corresponding to the one component is easily computed opening up the other component to a braid shown in Figure~\ref{fig:braid monodromy}. The braid allows one to determine the action of the monodromy on a graph connecting the punctures, which is a retract of the punctured disk (see \cite{boyland1994methods}).
This yields that $\phi(a)$ covers $a$ twice and each of $b$ and $c$, $\phi(b)$ covers each of $a,b,c$ once, and $\phi(c)$ covers $a$ and $b$.

\begin{figure}[ht]
    \centering
    \begin{overpic}[width=4cm]{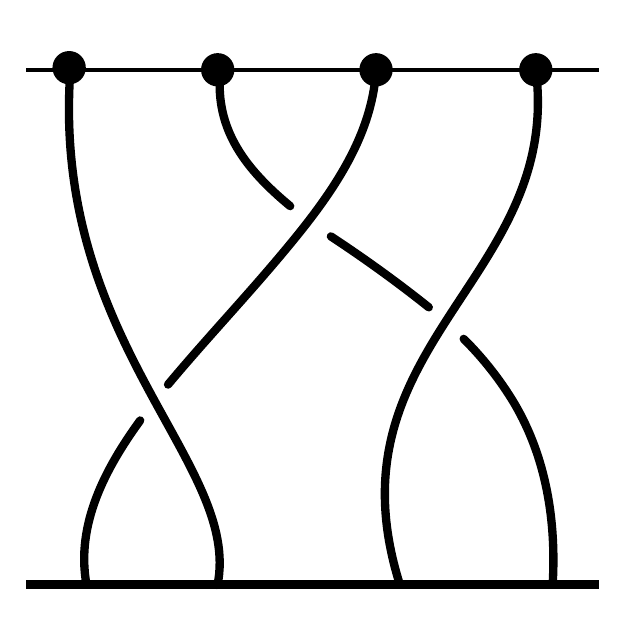}
    \put(20,95){$e_1$}
    \put(45,95){$e_2$}
    \put(75,95){$e_3$}
        \end{overpic}
        \begin{overpic}[width=5cm]{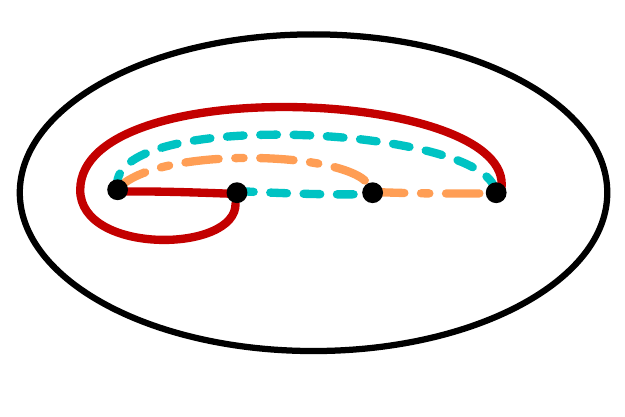}
    \put(25,20){$e_1$}
    \put(46,20){$e_2$}
    \put(67,20){$e_3$}
        \end{overpic}      
    \caption{The braid corresponding to the arcs of $K$ in the complement of a disk bounded by $J$, the segments connecting them inducing the monodromy on the punctured disk, depicted on the right.}
    \label{fig:braid monodromy}
\end{figure}
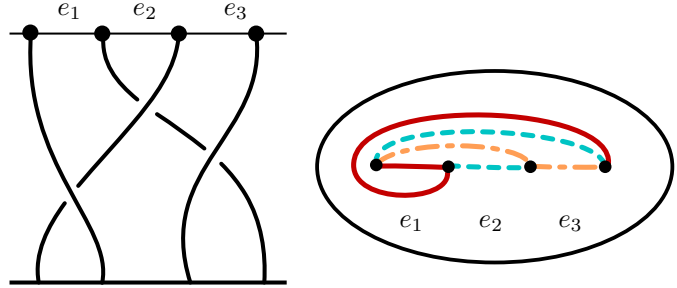

When looking at the branch cover in Figure~\ref{fig:disks}, the four disks connected by three twisted bands on the left end up as four disks connected by 6 bands. We can simplify by cut and paste, and this results in the twice punctured torus as depicted in Figure~\ref{fig:1-1 monodromy}.

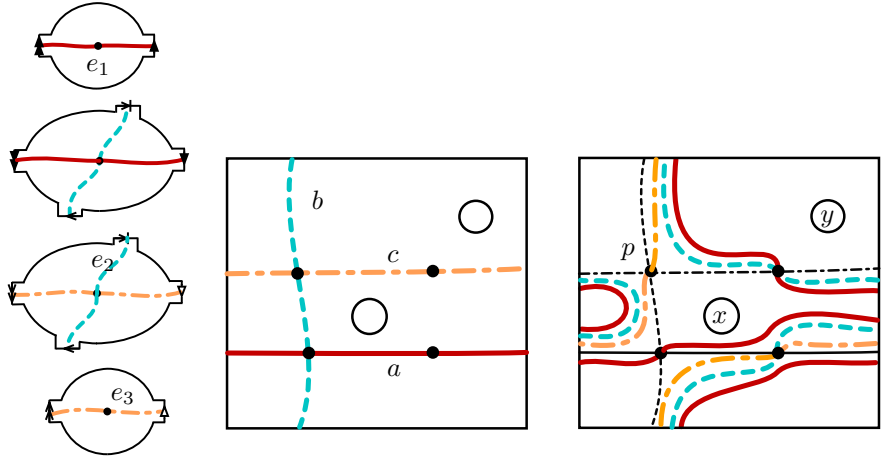
\begin{figure}[ht!]
    \centering
    \begin{overpic}[width=2.6cm]{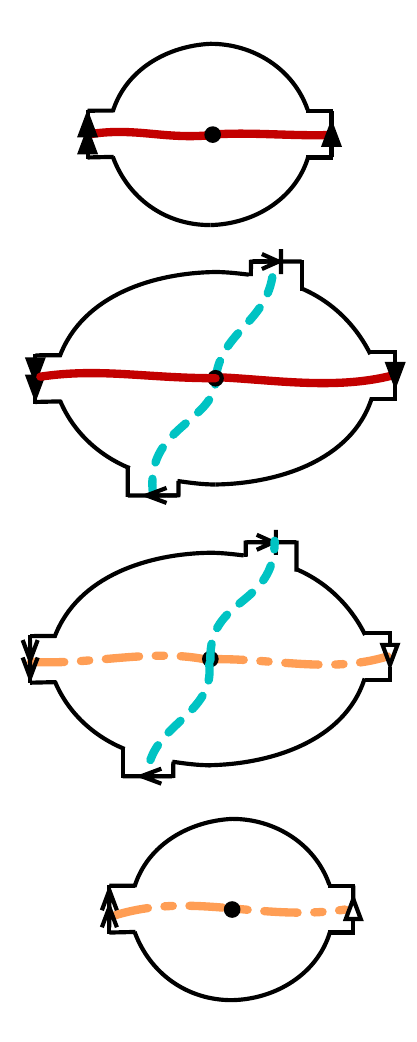}
    \put(18,82){$e_1$}
    \put(19,42){$e_2$}
    \put(23,15){$e_3$}
        \end{overpic}
         \begin{overpic}[width=10cm]{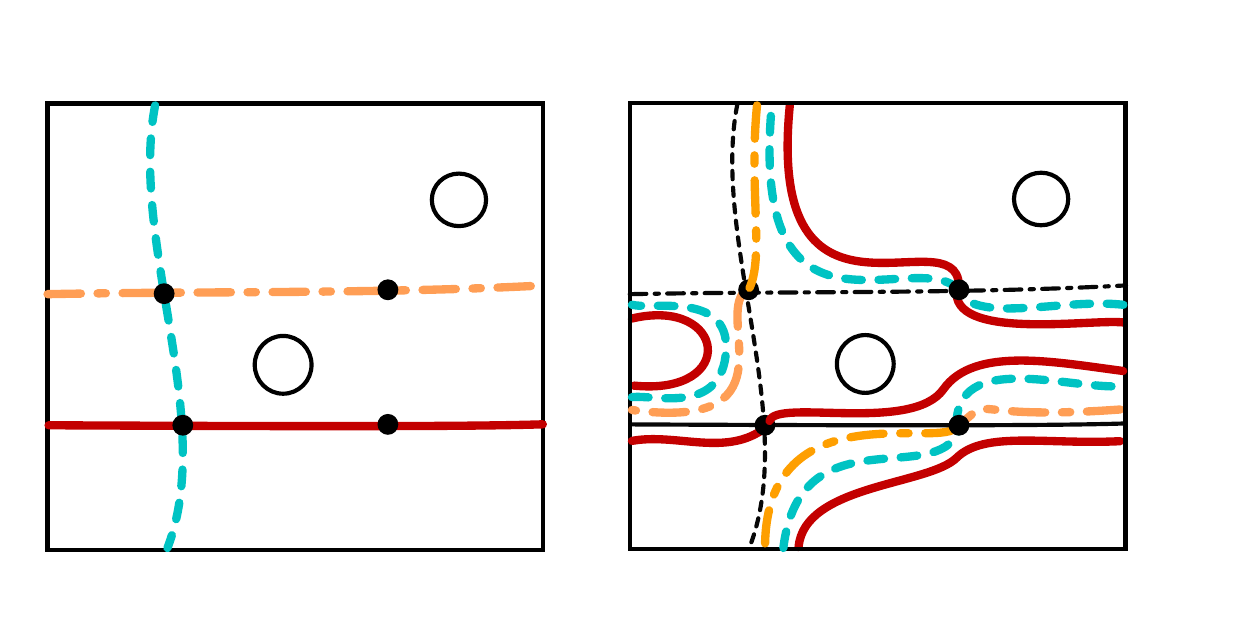}
    \put(25,13){$a$}
    \put(15,35){$b$}
    \put(25,28){$c$}
    \put(56,29){$p$}
    \put(68,20){$x$}
    \put(82.3,33.3){$y$}
        \end{overpic}
    \caption{The twice punctured torus as on the right of Figure~\ref{fig:disks}, branch covering the disk on the left of figure \ref{fig:disks}. Keeping track of the simple closed curves $a,b,c$ double covering the edges $e_1, e_2 ,e_3$ of the graph respectively, one obtains the monodromy induced on the twice punctured torus.}
    \label{fig:1-1 monodromy}
\end{figure}

The monodromy of the Whitehead link is then up to isotopy given by the image of the graph created by the closed curves $a,b$ and $c$ which each cover twice the edges connecting the branch-points in the disk. The resulting monodromy is shown in Figure~\ref{fig:1-1 monodromy}. By considering only the simple closed curves $b$ and $c$ we see that the homeomorphism $f$ of the twice punctured torus is of shear type, taking the vertical loop $b$ to itself plus a loop on its right around one of the boundary components denoted by $x$, and, as seen by sliding the point $p$ to the left, taking $c$ to itself after a Dehn twist along $f(b)$. Furthermore, the three curves $a, b$ and $c$ bound a disk containing only the boundary component $x$ and there is a unique arc up to isotopy connecting it to the loop $c$, and similarly the image of this arc is contained in the image of this once punctured disk and thus connects $c$ to the boundary component $x$ trivially.  Therefore, the resulting monodromy is that of the simple pair $\frac{0}{1}\vee \frac{1}{1}$

\begin{figure}[ht!]
    \centering
    \begin{overpic}[width=12cm]{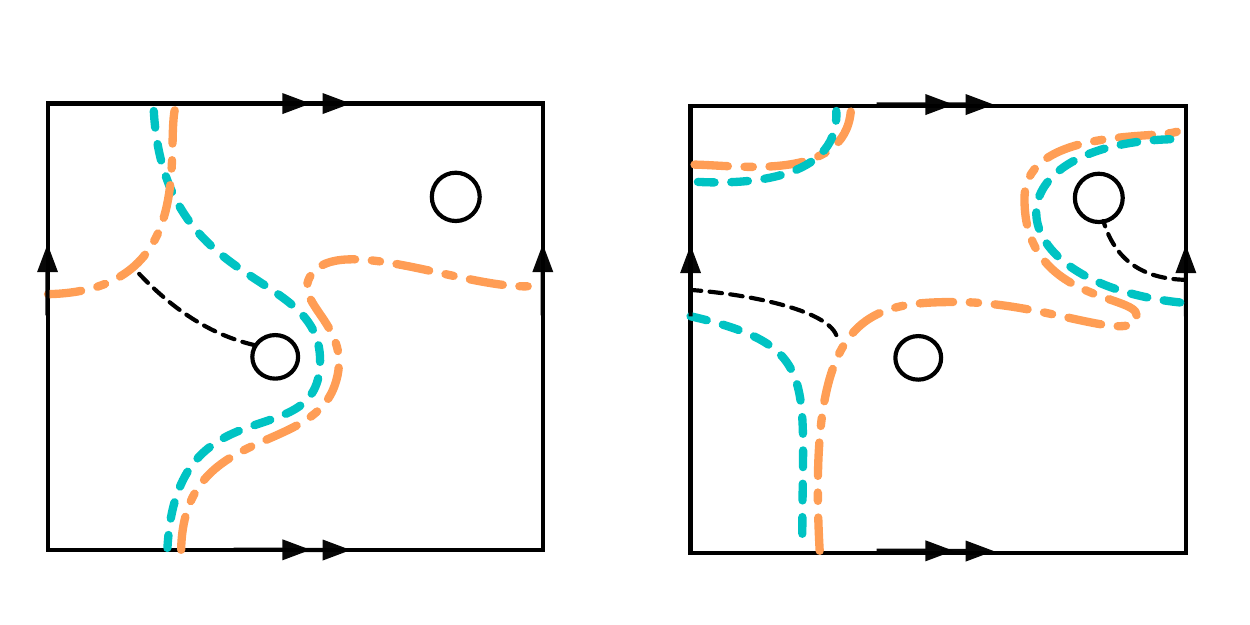}
    \put(18,0){Form $I$}
    \put(70,0){From $II$}
    \put(9,13){$f(b)$}
    \put(30,25){$f(c)$}
       \put(57.5,17){$f(b)$}
    \put(78,22){$f(c)$}
        \end{overpic}
    \caption{The two forms of the monodromy of the simple pair $\frac{0}{1}\vee \frac{1}{1}$. These forms are isotopic, pulling the intersection points of $f(b)$ and $f(c)$ and the top/bottom boundary of the fundamental domain. The dashed arcs are there to signal which of the two orbits is considered as the $x$ orbit in each case.}\label{fig:1-1 symmetry}
\end{figure}
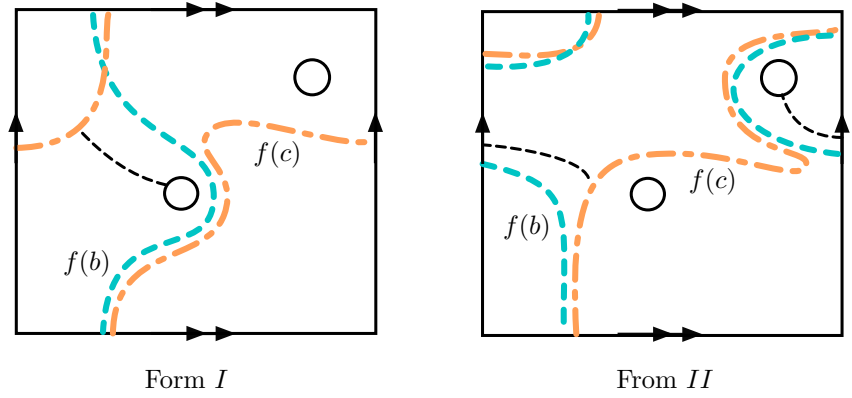
\end{proof}

\begin{remark}
The simple pair $\frac{0}{1}\vee \frac{1}{1}$ is symmetric, as both $x$ and $y$ can serve as the shorter orbit in Definition~\ref{def:simple_pair}. In Figure~\ref{fig:1-1 symmetry}, both isotopic forms of the $\frac{0}{1}\vee \frac{1}{1}$ are shown, one corresponding to the first form of a simple pair when the once punctured disk is on the right of a vertical loop of the graph, and one corresponding to the form where it is on the left of a vertical loop.
\end{remark}

\begin{thm}\label{thm: maonodromies are simple pairs}
Any monodromy arising from a fibration of the Whitehead link exterior corresponds to a simple pair.
\end{thm}

\begin{proof}
The fact that we now know by Proposition~\ref{prop:The 1-1 case} one monodromy of the Whitehead link complement allows us to depict this manifold as a mapping torus over the twice punctured torus as in Figure~\ref{fig:suspension}.

\begin{figure}[!ht]
    \centering
    \begin{overpic}[width=8cm]{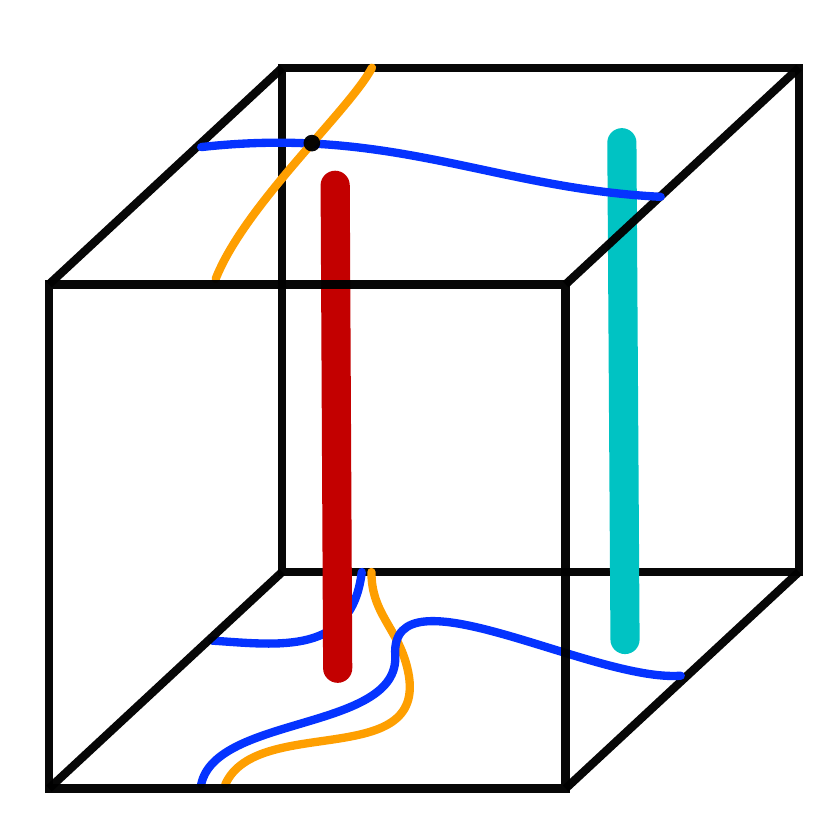}
    \put(59,81){$a$}
    \put(44,86){$b$}
    \put(55,18){$f(a)$}
    \put(45,8){$f(b)$}
        \end{overpic}
    \caption{The Whithead link complement as a mapping torus, where the top is glued to the bottom by the monodromy found in Proposition~\ref{prop:The 1-1 case}, the curve $a$ glued to $f(a)$ and the curve $b$ glued to $f(b)$. The front face is identified with the back face and the right face with the left face, as induced by the identifications on the edges of the square in Figure~\ref{fig:1-1 monodromy}.}
    \label{fig:suspension}
\end{figure}
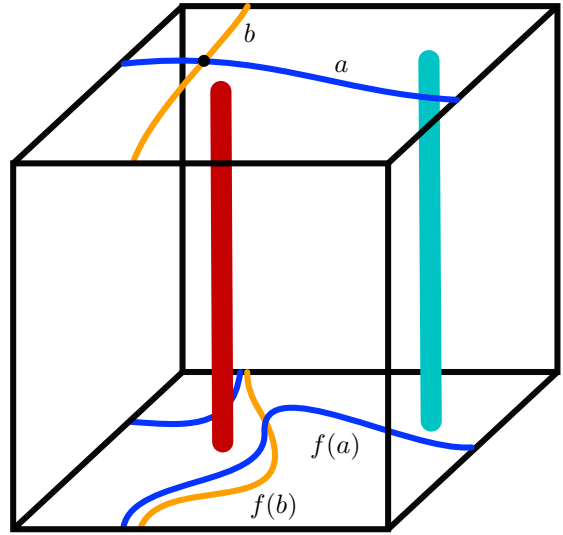

This representation allows as to identify two once punctured tori embedded in the Whitehead link complement, depicted in Figure \ref{fig:two_surfaces}: 
\begin{enumerate}
    \item A torus $\lambda_1$ with a single puncture on one of the orbits, a front edge that is glued to its back edge, and an edge on the top of the cube along the edge $b$, that is glued to its bottom edge that is along $f(b)$.
    
    \item Similarly, a torus $\lambda_2$ with a single puncture on the second orbit, that also has top edge on $b$ and bottom edge on $f(b)$.
\end{enumerate} 

It follows that these are the same surfaces obtained by adding to each a tube to each of the thrice punctured spheres $P_1$ and $P_2$ depicted in Figure~\ref{fig_whitehead}.

\begin{figure}[ht!]
    \centering
    \begin{overpic}[width=6cm]{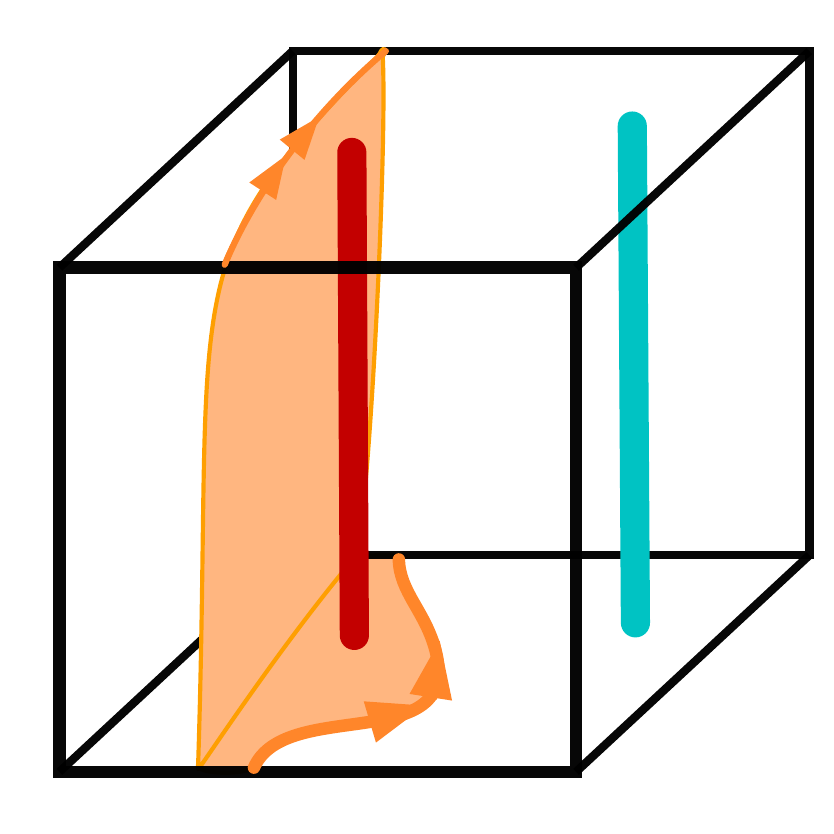}
    \put(26,75){$b$}
    \put(53,13){$f(b)$}
    \end{overpic}
        \begin{overpic}[width=6cm]{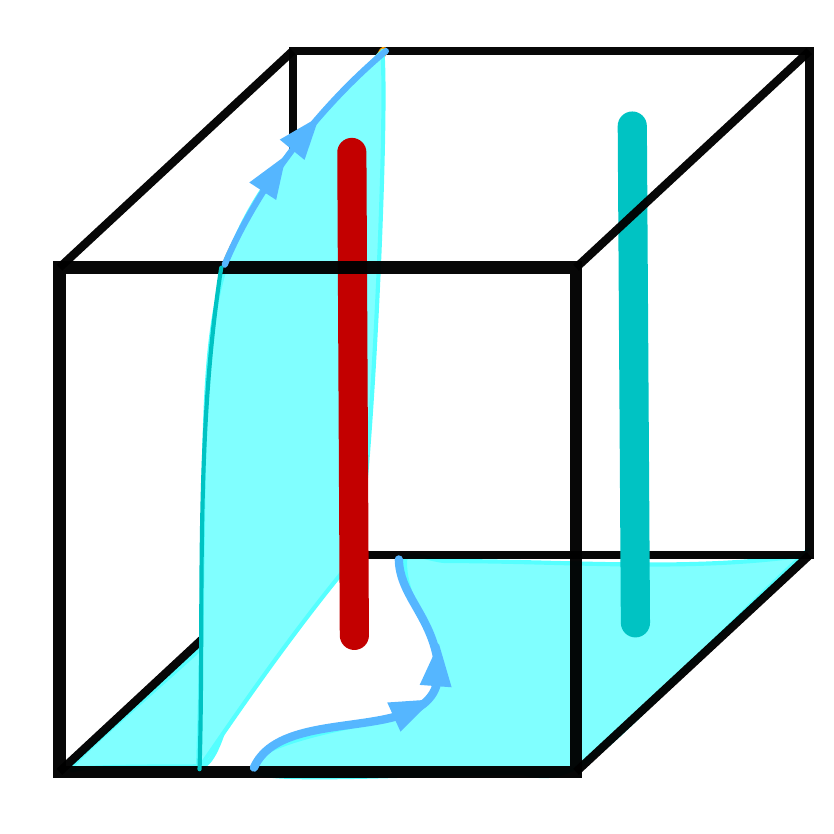}
    \put(26,75){$b$}
    \put(53,13){$f(b)$}
    \end{overpic}
    \caption{The two punctured tori, $\lambda_1$ and $\lambda_2$.}
    \label{fig:two_surfaces}
\end{figure}

Thus, it follows from \cite{Thurston1986Norm} that the integral second homology of the Whitehead link complement is generated over $\Z$ by $\lambda_1$ and $\lambda_2$, and for any two relatively prime numbers $k$ and $\ell$, $k,\ell \neq0$, the vector $(k,\ell)$ intersects the positive face of the Thurston norm ball in an internal rational point. Therefore, there is a fibration of the Whitehead link complement with a fiber given by the Haken sum $k\lambda_1+\ell\lambda_2$.

To understand the surface
$\lambda_{k,\ell}=k\lambda_1+\ell\lambda_2$ 
and its monodromy, first note that $\lambda_1$ and $\lambda_2$ intersect along their $b\sim f(b)$ edge. Orienting both so that the normal on the flat pieces in figure \ref{fig:two_surfaces} points upwards, one can directly compute that the vertical parts cancel out and the resulting surface $\lambda_{1,1}$ is the twice punctured torus $T$ that is the bottom horizontal face of the fundamental domain in Figure \ref{fig:suspension}.
Thus, If $k>\ell$ then $\lambda_{k,\ell}=(k-\ell)\lambda_1 + \ell T$,
and if $k<\ell$, $\lambda_{k,\ell}=k T + (\ell-k)\lambda_2$.
Thus each of the surfaces that arise as a fiber of some fibration can be seen as taking a few copies of $T$ and adding a few copies of either $\lambda_1$ or $\lambda_2$.
Figure~\ref{fig:schematic} shows the situation schematically.

\begin{figure}[ht!]
    \centering
    \begin{overpic}[width=6cm]{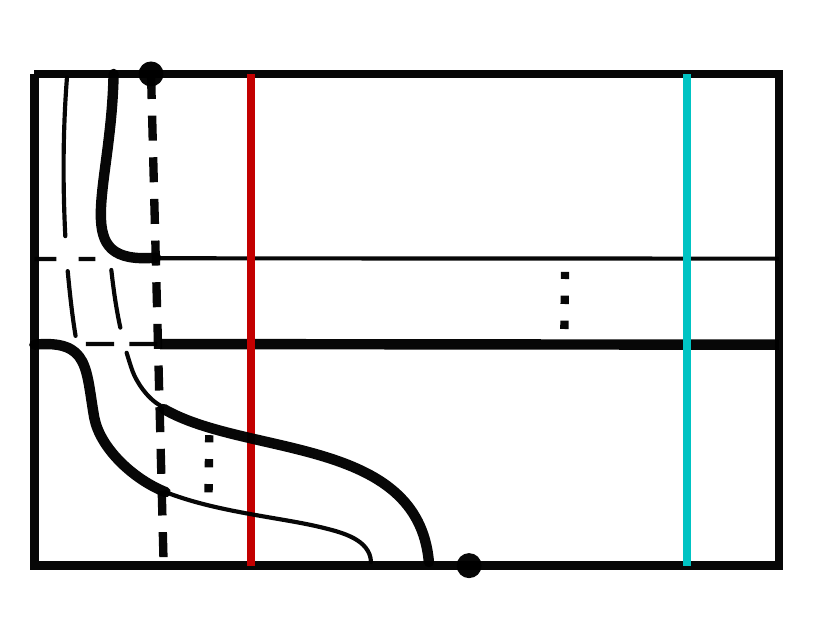}
    \put(14,70){$b$}
    \put(55,11){$f(b)$}
    \put(38,6){$\underbrace{}_{(k-\ell)\lambda_1}$}
    \put(-10,36){$\ell T \left\{\right.$}
        \end{overpic}
            \begin{overpic}[width=6cm]{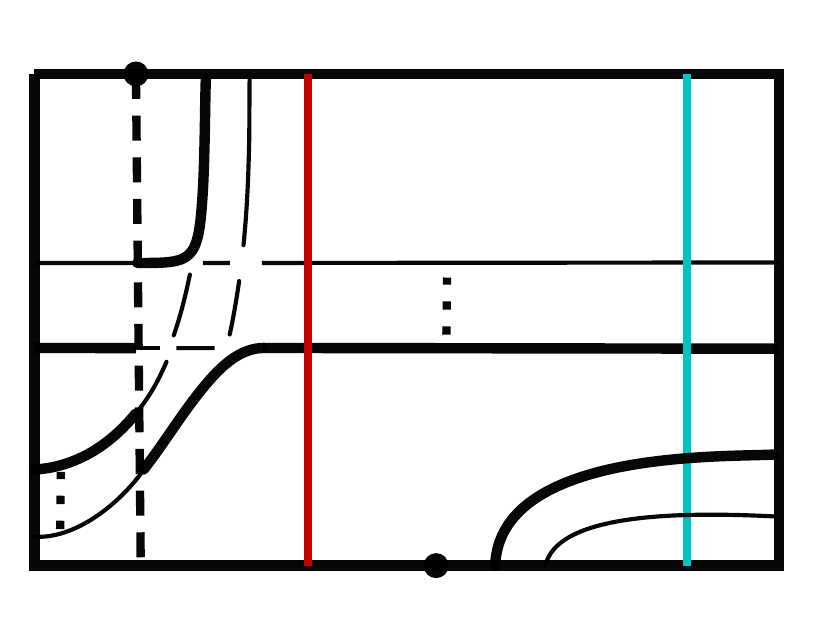}
   \put(14,70){$b$}
    \put(47,11){$f(b)$}
    \put(55,6){$\underbrace{}_{(\ell-k)\lambda_2}$}
    \put(94,36){$\left.\right\}kT $}
        \end{overpic}       
    \caption{The surface $\lambda_{k,\ell}=k\lambda_1+\ell\lambda_2$,  on the right in the case $p>q$, and on the left the case $q>p$. This is a side view into the cube in Figure~\ref{fig:suspension}. The vertical dashed line is the vertical plane through the $b$ vertical edge edge $b$ on the twice punctured torus in the bottom or top of the cube. The point $f(b)$ represents the image of this curve by the diffeomorphism $f$ gluing the top of the cube to the bottom.}
    \label{fig:schematic}
\end{figure}

\begin{figure}[ht!]
    \centering
    \begin{overpic}[width=6cm]{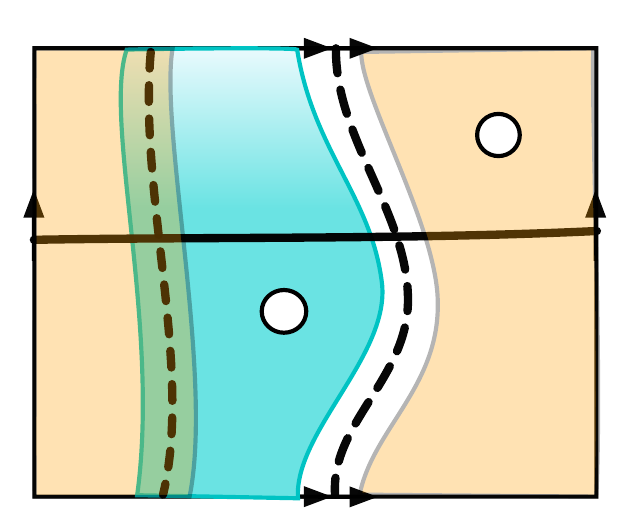}
    \put(38,56){$\lambda_1$}
    \put(80,20){$\lambda_2$}
        \end{overpic}
    \caption{The twice punctured torus $T$ and the two subsurface corresponding to $\lambda_1$ and $\lambda_2$.}
    \label{fig:subsurfaces}
\end{figure}

The horizontal and vertical edges on $T$ induce a horizontal and vertical edge on each of the surfaces by flowing them upwards through the manifold in its decomposition into a mapping torus, until they reach the surfaces as in Figure~\ref{fig:two_surfaces}. The resulting pieces are depicted in Figure \ref{fig:pieces}

\begin{figure}[ht!]
    \centering
    \begin{overpic}[width=11cm]{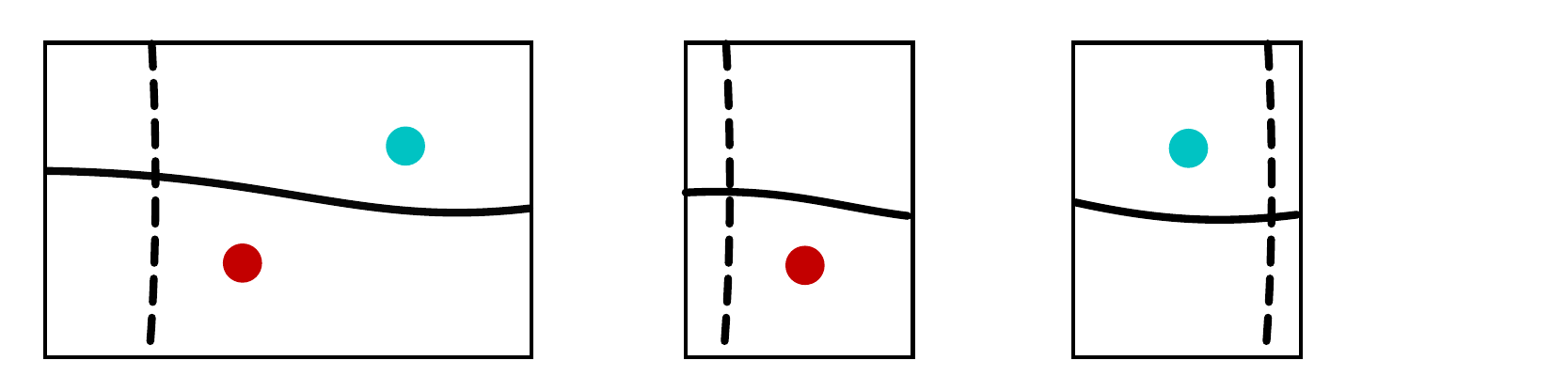}
    \put(20,-2){$T$}
    \put(51,-2){$\lambda_1$}
    \put(75,-2){$\lambda_2$}
        \end{overpic}
    \caption{The edges induces on the pieces $T,\lambda_1,\lambda_2$ making up any $\lambda_{k,\ell}$ surface.}
    \label{fig:pieces}
\end{figure}

When we perform the Haken sum and glue the $T,\lambda_1,\lambda_2$ pieces to a single surface, the horizontal edges connect to a single horizontal simple closed curve winding through all the pieces, while the vertical edges on each of the pieces remain separate simple closed curve. 

\begin{figure}[ht!]
    \centering
    \begin{overpic}[width=15cm]{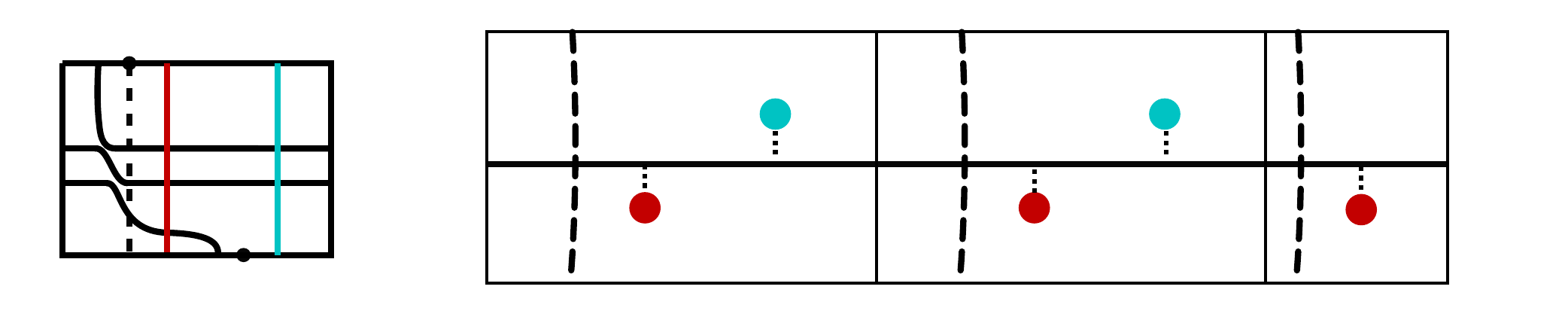}
    \put(12,1){$R_3$}
    \put(22,10.5){$R_2$}
    \put(22,7){$R_1$}
    \put(43,19){$R_1$}
    \put(68,19){$R_2$}
    \put(86,19){$R_3$}
    \end{overpic}
    \caption{The surface $2T+\lambda_1=3\lambda_1+2\lambda_2=\lambda_{2,3}$. It was proven in \cite{pinskyWajnryb} that one can connect the horizontal lines to peripheral loops encirceling the orbits, so that they are mapped straightforwardly to each other without winding around the other periodic point in the rectangle.}
    \label{fig:puzzle}
\end{figure}

The monodromy on this surface is obtained by sliding the surface upwards through the manifold represented as in Figure~\ref{fig:suspension} until it returns to itself. Sliding upwards induces the following action on the edges:

\begin{itemize}

    \item Each vertical edge in a rectangle to the vertical edge above it along the dashed line in Figure~\ref{fig:schematic}, except the topmost vertical edge $l$ (that is in the rectangle marked $R_{top}$ in bold on the top of the figure) that contains a single orbit, that is mapped to a closed curve $f(l)$ in the bottommost rectangle ($R_{bottom}$ in the figure) that separates the two points belonging to different orbits there. This curve was not previously a vertical edge but is equal to the vertical edge on its right or left together with a small loop around one of the points of the periodic orbits there. The loop is on the right if $k>l$ in which case a $\lambda_1$ piece appears bottommost in the the embedding of $k\lambda_1+\ell\lambda_2$ into the mapping torus (as in Figure~\ref{fig:schematic}) and the image of the monodromy in this $\lambda_1$ piece is given by  the left side on Figure~\ref{fig:monodromy for bottom piece}. The loop is on the left if $\ell>k$ and there is a bottommost $\lambda_2$ piece, on which the image of the monodromy is given in the right side of Figure~\ref{fig:monodromy for bottom piece}. 

    \item Similarly, the horizontal segment in each edge is mapped to the horizontal edge above it in another rectangle composing the surface, except the topmost horizontal segment that goes to the horizontal segment in the bottommost rectangle with a Dehn twist along $f(l)$.
    
\end{itemize}

\begin{figure}[ht!]
    \centering
    \begin{overpic}[width=8cm]{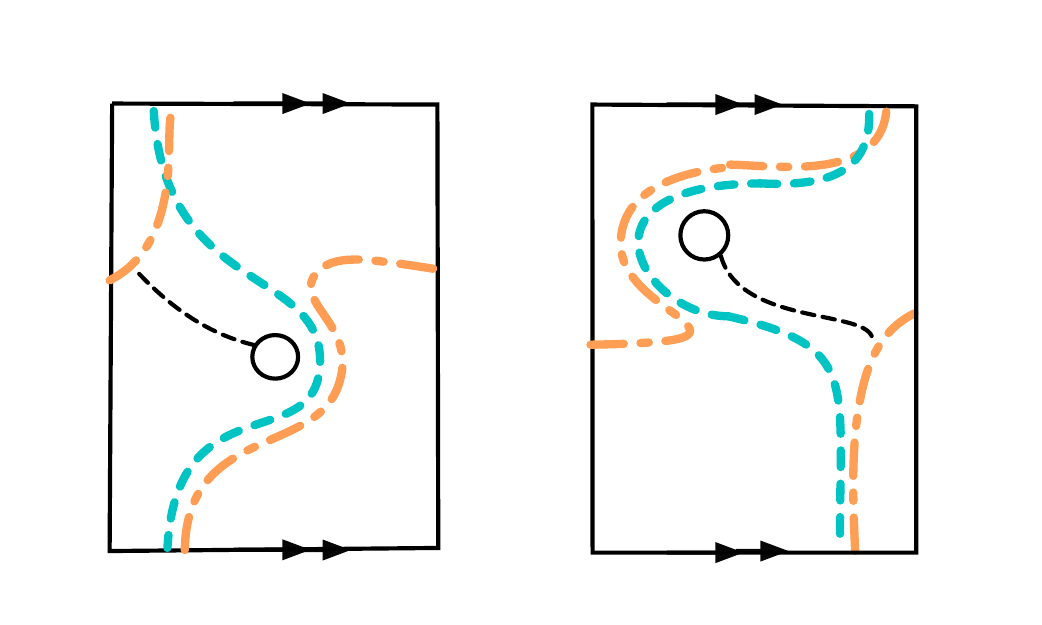}
    \put(23,1){$\lambda_1$}
    \put(70,1){$\lambda_2$}
        \end{overpic}
    \caption{The image of the monodromy on the bottom piece of the surface, according to whether it is $\lambda_1$ (when $k>\ell$) or $\lambda_2$ when ($k<\ell$).}
    \label{fig:monodromy for bottom piece}
\end{figure}

As by Definition~\ref{def:simple_pair} this action on a graph implies $x\vee y$ is a simple pair, this completes the proof.
\end{proof}

By Theorem~\ref{thm: maonodromies are simple pairs} the set of monodromies of fibrations are a subset of the set of all simple pairs homeomorphisms, thus to prove Theorem \ref{thm: 1-1 monodromies and simple pairs} it remains to show the other direction:

\begin{thm}
\label{thm: Any pair is a monodromy}
For any pair of Farey neighbors $\frac{p}{q}$ and $\frac{r}{s}$, there is a simple pair monodromy $\frac{p}{q}\vee\frac{r}{s}$ that is obtained as a monodromy of some fibration of the Whitehead link complement.
\end{thm}

\begin{proof}
    Choose $q$ and $s$ relatively prime. There is a unique solution to each of the equations
    \[
ps-rq=1,\  p's-r'q=-1
    \]
    
so that $0<p<q$ and $0<r<s$, thus there are exactly two pairs that are Farey neighbors with given denominators $q$ and $s$. The solution $p'$ and $r'$ to the second equations satisfies $p'=q-p$ and $s'=r-s$ when $p,s$ is the solution to the first equation.
Therefore, in one of these pairs the fraction with denominator $q$ is the larger number and in the other it is smaller than the fraction with denominator $s$.

Next, assuming $q>s$ and denoting the longer orbit by $x$ as in Section~\ref{sec:Pairs}, the $x$ orbit has $q$ points and the $y$ orbit $s$ points.
Consider the fibration of the Whitehead link complement corresponding to $q\lambda_1+s\lambda_2$.
The $x$ orbit intersects the $q$ $\lambda_1$ surfaces, 
and the closed curve $f(l)$ in the rectangle $R_1$
 separates the $x$ orbit on its left from the $y$ orbit on its right. Therefore, it is isotopic in the complement of the orbits to the edge on the right of $R_1$ with a small loop around the shorter $y$ orbit to its left. Thus, as there is a point of the $x$ orbit in each rectangle and the vertical edges thus have the same rotation number of the $x$ orbit, this corresponds to the case 
the $y$ orbit with rotation number $r/s$ is rotating a little faster than the $x$ orbit, i.e.
$\frac{r}{s}>\frac{p}{q}$, or $ps-rq=-1$.
In the case $s\lambda_1+q\lambda_2$, the longer $x$ orbit intersects the $\lambda_2$ surfaces, and $f(l)$ is isotopic to the left edge or $R_1$ plus a loop to its right. That is, the $y$ orbit with rotation number $r/s$ is rotating a little slower than the $x$ orbit, i.e. 
$\frac{r}{s}<\frac{p}{q}$, or $ps-rq=1$.

Thus we find both pairs appear as fibrations of the Whitehead link complement, and therefore all possible pairs appear as the monodromy of some fibration.
\end{proof}

Note that the equivalence between the simple orbit corresponding to $\frac{p}{q}$ and the one corresponding to $1-\frac{p}{q}$ can now be understood as the result of the symmetry between the two components of the Whitehead link.

\begin{remark}
The above proof does not give much information about the ordering of the $T,\lambda_1$ and $\lambda_2$ pieces within a fiber $S=k\lambda_1+\ell\lambda_2$. One way to determine explicitly the order in which the pieces are glued 
is given by the Christoffel words. For example Figure~\ref{fig:christoffel} shows two Christoffel words corresponding to the surface $S=7\lambda_1+4\lambda_2=4 T+3 \lambda_1$. They both give the same order of pieces up to cyclic order along the horizontal direction.
\end{remark}

\begin{figure}[ht!]
    \centering
   \begin{overpic}[width=6cm]{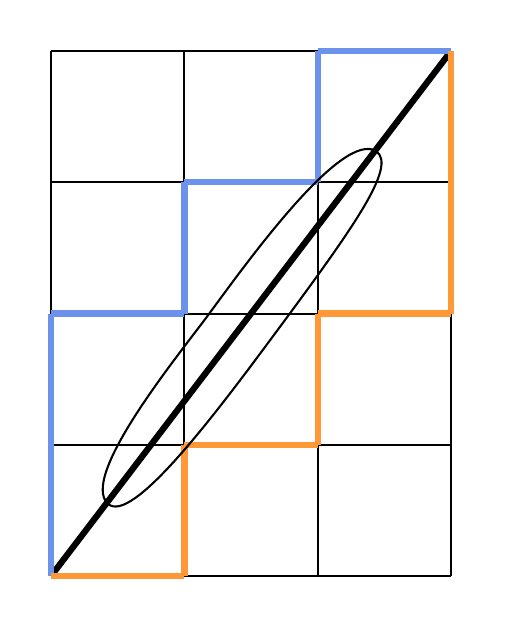}
   \put(60,2){$\lambda_1$}
   \put(2,79){$T$}
   \put(18,9){0}
   \put(31,15){1}
   \put(40,30){0}
   \put(52,38){1}
   \put(60,51){0}
   \put(73,59){1}
   \put(73,80){1} 
   \put(2,15){1}
   \put(2,38){1}
   \put(18,51){0}
   \put(26,59){1}
   \put(40,72){0}
   \put(47,80){1}
   \put(60,93){0} 
   \put(20,27){1}
   \put(27.5,37.5){0}
   \put(35,48){1}
   \put(49,65){0}
   \put(52.5,69.5){1}
   \end{overpic}
 \caption{The image of the monodromy on the bottom piece of the surface, according to whether it is $\lambda_1$ (when $k>\ell$) or $\lambda_2$ when ($k<\ell$).} \label{fig:christoffel}
\end{figure}
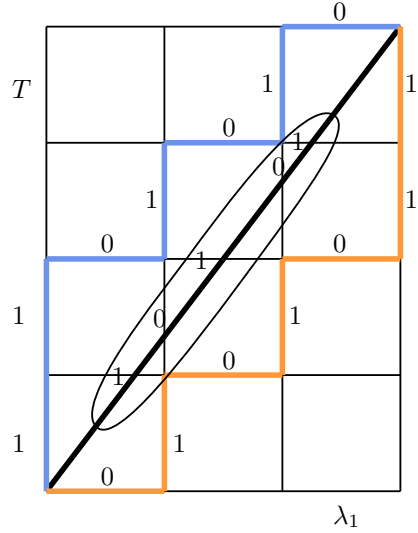

\begin{cor}
For a primitive integral class $(k, \ell) \in \mathcal{C}$, 
the stable foliation $\mathcal{F}_{(k, \ell)}$ of the monodromy $\Phi_{(k,\ell)}$ of the fibration on the Whitehead link exterior 
$W$  has the following properties. 
\begin{itemize}
\item 
$\mathcal{F}_{(k+\ell)}$ is $1$-pronged at each boundary component of the fiber $F_{(k,\ell)}$. 

\item 
$\mathcal{F}_{(k+\ell)}$ has $(k+\ell)$  $3$-pronged singularities in the interior of the fiber $F_{(k,\ell)}$. 
\end{itemize}
\end{cor}

\begin{proof}
    Figure \ref{fig:train track} shows a train track construction for the simple pair $0\vee1$, with one 1-prong and two 3-prong singularities. This induces a train track structure on all other fibrations by the construction in Theorem~\ref{thm: maonodromies are simple pairs}, with singularities as required.

\begin{figure}[ht!]
    \centering
    \includegraphics[width=10cm]{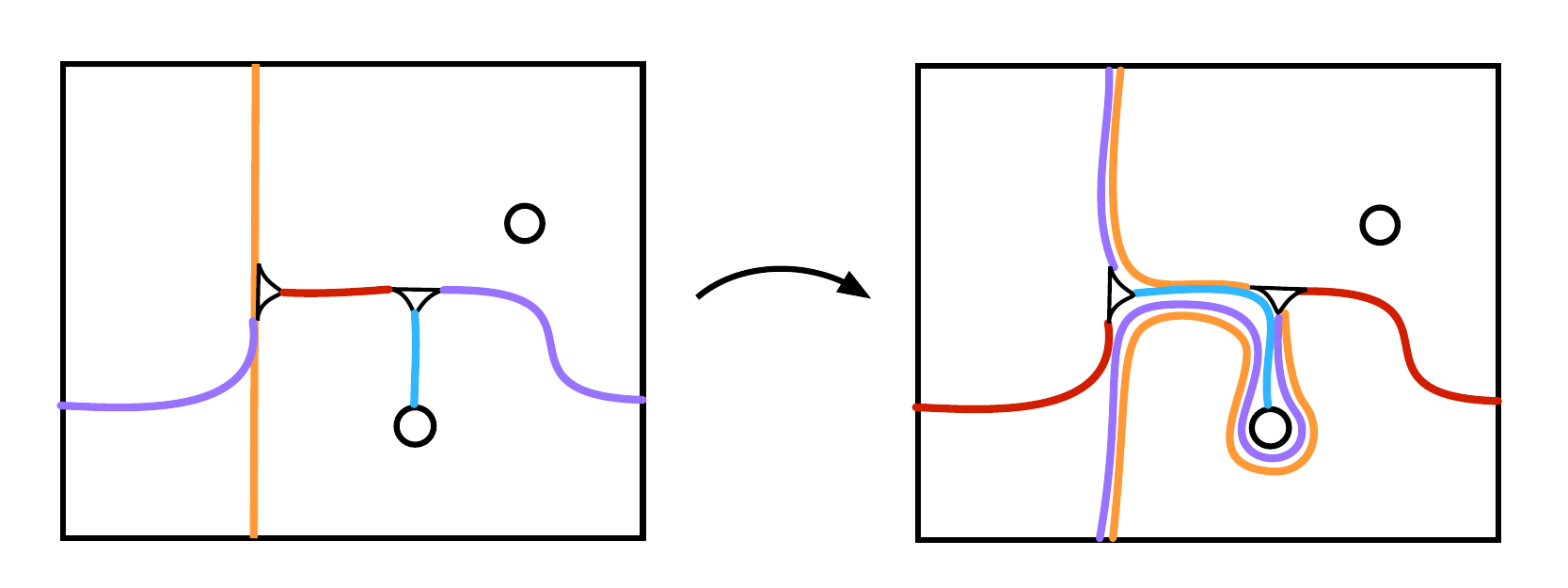}
    \caption{The train track for $0\vee1$.}
    \label{fig:train track}
\end{figure}
\end{proof}

\bibliographystyle{plain}
\bibliography{Bibliography.bib}

\end{document}